\documentclass[12pt]{article}

\usepackage{amssymb}%
\usepackage{amsmath}%
\usepackage{amsthm}%
\usepackage{url}

\usepackage{fullpage}%
\usepackage{xcolor}%

\allowdisplaybreaks%

{\theoremstyle{plain}%
 \newtheorem{theorem}{Theorem}
 \newtheorem{corollary}{Corollary}

}
{\theoremstyle{remark}

}
{\theoremstyle{definition}

\newtheorem{example}{Example}
}

\begin{document}

\begin{center}
{\large Hyperbolic summations derived using the Jacobi functions dc and nc}

 \ 

 John M.\ Campbell
\end{center}

\begin{abstract}
 We introduce a method that is based on Fourier series 
 expansions related to Jacobi elliptic functions and that we apply to determine new identities for evaluating 
 hyperbolic infinite sums in terms of the complete elliptic integrals $K$ and $E$. We apply our method to determine generalizations of a family of 
 $ \text{sech}$-sums given by Ramanujan and generalizations of a family of $\text{csch}$-sums given by Zucker. Our method has the advantage 
 of producing evaluations for hyperbolic sums with sign functions that have not previously appeared in the literature on hyperbolic sums. We apply 
 our method using the Jacobian elliptic functions $\text{dc}$ and $\text{nc}$, together with the elliptic alpha function, to obtain new closed forms for 
 $q$-digamma expressions, and new closed forms for series related to discoveries due to Ramanujan, Berndt, and others. 
\end{abstract}


\vspace{0.05in}


\section{Introduction}
The Jacobi elliptic functions are defined via the inversion of the elliptic integral 
\begin{equation}\label{actualF}
 u = F(\phi, k) = \int_{0}^{\phi} \frac{dt}{\sqrt{1 - k^2 \sin^2 t}} 
\end{equation}
 of the first kind, letting $0 < k^2 < 1$. The expression 
 $k = \text{mod} \, u$ is referred to as the \emph{elliptic modulus}, and $\phi = \text{am}(u, k) = 
 \text{am}(u)$ is referred to as the \emph{Jacobi amplitude}. The \emph{Jacobi elliptic functions} 
 $\text{sn}$, $\text{cn}$, and $\text{dn}$ 
 may be defined as follows: 
\begin{align}
 \sin \phi & = \text{sn}(u, k), \label{definitionsn} \\
 \cos \phi & = \text{cn}(u, k), \ \ \ \text{and} \label{definitioncn} \\ 
 \sqrt{1 - k^2 \sin^2 \phi} & = \text{dn}(u, k). \label{definitiondn} 
\end{align}
 In this article, we apply Fourier series 
 expansions related to Jacobi elliptic functions to build on the work on 
 hyperbolic infinite sums due to Zucker \cite{Zucker1979}, 
 Yakubovich \cite{Yakubovich2018}, and Xu \cite{Xu2018}. 
 The relevance of 
 Zucker's work in \cite{Zucker1979} 
 (cf.\ \cite{Zucker1984}) within high-energy physics \cite{BourdierDrukkerFelix2016,BourdierDrukkerFelix2015,Drukker2015}, nuclear physics 
 \cite{DowkerKirsten2002}, electrostatics \cite{Lekner2010}, the study of lattice sums \cite{BorweinGlasserMcPhedranWanZucker2013}, number-theoretic 
 subjects in the study of Fibonacci sums 
 \cite{DuverneyNishiokaNishiokaShiokawa1997,DuverneyShiokawa2008,ElsnerShimomuraShiokawa2007,ElsnerShimomuraShiokawa2012,ElsnerShimomuraShiokawa2008}, 
 and number-theoretic areas concerning Ramanujan's series for $\frac{1}{\pi}$ \cite{BagisGlasser2012,BagisGlasser2013,BorweinBorwein1987Indian} 
 and for constants such as Ap\'{e}ry's constant \cite{Vepstas2012} 
 serve as a main source of motivation behind the techniques and results introduced 
 in this article. 

\subsection{Background and further preliminaries}
 The complete elliptic integral of the first kind $K$ is such that 
 $K(k) = F\left( \frac{\pi}{2}, k \right)$. 
 We may also write $K(k) = K$ and $K' = K'(k) = K(k')$, with $k' = \sqrt{1 - k^2}$, 
 and similarly with respect to the complete elliptic integral
 $ E(k) = \int_{0}^{\frac{\pi}{2}} \sqrt{1 - k^2 \sin^{2} \theta} \, d\theta $ 
 of the second kind. 

 Combinations involving the $\text{sn}$, $\text{cn}$, and $\text{dn}$ functions give us the remaining Jacobi elliptic functions, as listed below: 
\begin{align}
 \text{cd}(u, k) 
 & = \frac{\text{cn}(u, k)}{\text{dn}(u, k)}, \label{cddefinition} \\ 
 \text{cs}(u, k) 
 & = \frac{\text{cn}(u, k)}{\text{sn}(u, k)}, \nonumber \\ 
 \text{dc}(u, k) 
 & = \frac{\text{dn}(u, k)}{\text{cn}(u, k)}, \label{dcdefinition} \\ 
 \text{ds}(u, k) 
 & = \frac{\text{dn}(u, k)}{\text{sn}(u, k)}, \nonumber \\
 \text{nc}(u, k) 
 & = \frac{1}{\text{cn}(u, k)}, \nonumber \\
 \text{nd}(u, k) 
 & = \frac{1}{\text{dn}(u, k)}, \nonumber \\
 \text{ns}(u, k) 
 & = \frac{1}{\text{sn}(u, k)}, \nonumber \\ 
 \text{sc}(u, k) 
 & = \frac{\text{sn}(u, k)}{\text{cn}(u, k)}, \ \ \ \text{and} \nonumber \\
 \text{sd}(u, k) 
 & = \frac{\text{sn}(u, k)}{\text{dn}(u, k)}. \nonumber 
\end{align}

  The $\Gamma$-function \cite[\S8]{Rainville1960} is famous and ubiquitous as a special function,   and is to frequently arise in our work.   We recall that this  
 special function  is defined for $\Re(x) > 0$  with the Euler integral  $\Gamma(x)=\int_{0}^{\infty} u^{x-1}e^{-u}\,du $ \cite[\S8]{Rainville1960}.   A 
 useful feature concerning our method in Section \ref{subsectiontechnique}  is given by how this method may be used to obtain new  evaluations for  
 expressions involving the $q$-digamma function $\psi_{q}(z)$, which may be defined so that  $$ \psi_{q}(z) = \frac{1}{\Gamma_{q}(z)}  
  \frac{\partial \Gamma_{q}(z)}{ \partial z}, $$  where $\Gamma_{q}$ denotes the $q$-analogue of the $\Gamma$-function. Equivalently, we may  define the  
  $q$-digamma function so that  
\begin{equation}\label{qdigamma}
 \psi_{q}(z) = -\ln(1 - q) + \ln q \sum_{n=0}^{\infty} \frac{q^{n+z}}{1 - q^{n+z}}. 
\end{equation}

 In Ramanujan's second notebook \cite[\S17]{Berndt1991} (cf.\ \cite{Berndt2016}), identities for evaluating 
\begin{equation}\label{Ramanujanmainmotivate}
 \sum _{n = 0}^{\infty} (-1)^n (2 n + 1)^{s} 
 \text{sech}\left(\frac{2 n + 1}{2} \frac{K'}{K} \pi \right) 
\end{equation}
 in terms of $K$ are given, such as the identity 
 \cite[p.\ 134]{Berndt1991} 
\begin{equation}\label{instanceRamanujan}
 \sum _{n = 0}^{\infty} (-1)^n (2 n+1) 
 \text{sech}\left( \frac{2 n + 1}{2} \frac{K'}{K} \pi \right) 
 = \frac{2 k k' K^2}{\pi ^2}. 
\end{equation}
 In Zucker's seminal article on hyperbolic sums \cite{Zucker1979}, an identity 
 for evaluating \eqref{Ramanujanmainmotivate} was also included, 
 and \cite{Zucker1979} also provided identities for the sums given by 
 replacing $\text{sech}$ with $\text{csch}$ 
 in Ramanujan's sums in \eqref{Ramanujanmainmotivate}: 
\begin{equation}\label{Zuckermain}
 \sum_{n=0}^{\infty} (-1)^{n} 
 (2n+1)^{s} \text{csch}\left( \frac{2n+1}{2} \frac{K'}{K} \pi \right) = 2 J_{s}, 
\end{equation}
 with, for example, 
 $$ 4 J_{0} = \left( \frac{2 K }{\pi} \right) k \ \ \ \text{and} \ \ \ 
 4 J_{2} = \left( \frac{2 K}{\pi} \right)^{3} k (1 - k^2). $$ 
 Zucker's methods in \cite{Zucker1979} mainly relied on 
 double series manipulations together with expansions such as 
 $$ J_{s}(c) = J_{s} = \sum_{n=1}^{\infty} \frac{ (-1)^{n+1} (2n-1)^{s} q^{n - 1/2} }{1 - q^{2n-1}}, $$ writing 
\begin{equation}\label{nome}
 q = e^{- \pi \frac{K'}{K}} 
\end{equation}
 to denote the \emph{nome} for Jacobian elliptic functions. 
 In this article, we provide a method that may be used to 
 to evaluate the members of both of the families of generalizations of 
 \eqref{Ramanujanmainmotivate} and \eqref{Zuckermain} indicated as follows: 

\begin{enumerate}

\item The sums obtained by 
 replacing $\text{sech}$ (resp.\ $\text{csch}$) with the higher power $\text{sech}^2$ (resp.\ $\text{csch}^2$), 
 for all odd powers $s$, 
 in \eqref{Ramanujanmainmotivate} (resp.\ \eqref{Zuckermain}); and 

\item The sums given by replacing the sign function $(-1)^{n}$ 
 with the sign function $(-1)^{ \left\lfloor \frac{n}{2} \right\rfloor}$ 
 within the sums indicated in the preceding point. 

\end{enumerate}

 Our evaluation technique may be applied much more broadly, apart from the families of hyperbolic sums indicated above. 
 For example, our method may also be applied to produce new summations that resemble and are related 
 to the summations 
 $$ \sum_{n=1}^{\infty} \frac{n}{e^{2 \pi n} - 1} = \frac{1}{24} - \frac{1}{8 \pi} 
 \ \ \ \text{and} \ \ \ \sum_{n=1}^{\infty} \frac{n^{13}}{e^{2 n \pi} - 1} = \frac{1}{24} $$ 
 due to Ramanujan \cite[pp.\ 326, xxvi]{Ramanujan1927}. 
 These Ramanujan summations 
 have been explored by 
 authors such as Nanjundiah \cite{Nanjundiah1951} and Sandham \cite{Sandham1950}. 
 As in \cite{Nanjundiah1951}, we record that Ramanujan had given the first out of the above identities 
 in his seminal article on modular equations and approximations to $\pi$ \cite{Ramanujan1915}. 

 As in Yakubovich's article \cite{Yakubovich2018}, our work is a continuation of the methods due to Ling and Zucker \cite{Ling1988,Ling1975,Zucker1979}. 
 Yakubovich's symbolic forms as in 
 $$ \frac{\Gamma^4 \left(\frac{1}{4}\right)}{32 \pi ^4} + 
 \frac{\Gamma^8 \left(\frac{1}{4}\right)}{512 \pi ^6} - 
 \frac{1}{8 \pi ^2} = \sum _{n=1}^{\infty} n^2 \cosh (\pi n) \text{csch}^2(\pi n) $$ 
 for infinite series involving $\text{csch}^{2}$ 
 were highlighted as Corollaries in 
 \cite{Yakubovich2018} and motivate 
 our symbolic forms for $\text{csch}^{2}$-sums as in 
 Examples \ref{202211053AAAAAAAAAAAAAAA2252PLUYS2PMSTAR1A} and 
 \ref{20221102339AM1A} below. 

\subsection{Main technique}\label{subsectiontechnique}
 Our main technique may be summarized in the following manner, letting $\text{j}(u, k)$ denote a Jacobi elliptic function. 

\begin{enumerate}

 \item Start with the Fourier series expansion for $\text{j}(u, k)$ 
 or for some expression involving a Jacobi elliptic function, 
 or some manipulation of such Fourier series expansions such as a series expansion obtained via term-by-term 
 applications of a differential operator; 

 \item In order to use built-in CAS algorithms for reducing, if possible, derivatives of $q$-powers 
 with respect to the elliptic modulus and using hyperbolic functions, we need to 
 rewrite the summand of the series indicated in the previous step 
 so that any exponential expressions only appear in the denominator, 
 and we need to simplify the powers for any such expressions; 

 \item Enforce a substitution such as $u \mapsto 2 w K $ for a variable $w$; 

 \item Argue, if possible, that if $w$ were to be set to 
 some special value, the series obtained from the third step 
 would reduce to a closed-form evaluation; and 

\item If the resultant summand is non-vanishing, and if differentiating with respect to the elliptic modulus yields a summand that may be expressed in 
 terms of hyperbolic functions, then simplify the resultant summand. 

\end{enumerate}

\subsection{Organization of the article}\label{subsectionOrganization}
 The above technique may be applied to many out of 
 the 12 of the Jacobian functions among $\text{cd}$, $\text{cn}$, $\text{cs}$, 
 $\text{dc}$, dn, ds, nc, nd, ns, sc, sd, and sn. 
 However, we sometimes obtain equivalent results by applying our technique to different Jacobian elliptic functions. 
 For example, if we apply this technique to $\text{ns}$, then the results that we would obtain 
 would also be obtainable 
 by applying our technique to $\text{dc}$. 
 The results in this article in Sections \ref{SectionMain} and 
 \ref{Mainsectionnc} are devoted to the application of our technique to $\text{dc}$ and 
 and $\text{nc}$, and if we were to attempt to apply the same method to Jacobi elliptic functions other than 
 $\text{dc}$ and and $\text{nc}$, this method would either not be applicable, or 
 we would obtain the same results as in Sections \ref{SectionMain} and \ref{Mainsectionnc} below. 

 Our method may also be applied to Fourier series expansions for expressions other than the 12 Jacobi functions we 
 have listed. For example, an equivalent formulation of the $q$-expansion 
\begin{equation}\label{nssquared}
 \left(\frac{2 K }{\pi }\right)^2 \text{ns}^2(2 
 K w,k) 
 = \frac{ 4 K ( K - E )}{\pi ^2}+\csc ^2(\pi 
 w)-8 \sum _{n=1}^{\infty} 
 \frac{ n q^{2 n} \cos (2 n \pi w)}{1 - 
 q^{2 n}} 
\end{equation}
 for $\text{ns}^{2}$ which, as recorded in \cite{Zucker1979}, 
 was included in the classic text \cite[p.\ 535]{WittakerWatson1927} 
 and dates back to the work of Jacobi \cite{Jacobi1829}, 
 may be applied in accordance with the technique in 
 Section \ref{subsectiontechnique}. 
 This is briefly considered in Section \ref{sectionpowers}. 

\section{Applications of the Jacobi elliptic function $\text{dc}$}\label{SectionMain}
 The connections between the Jacobi elliptic functions and the generalized Bessel functions, as explored in 
 \cite{DattoliChiccoliLorenzuttaMainoRichettaTorre1992}, serve as a further source of motivation concerning our applications related to series expansions 
 given in \cite{DattoliChiccoliLorenzuttaMainoRichettaTorre1992}. To begin with, we record the Fourier series expansion 
\begin{equation}\label{qdc}
 \text{dc}(u, k) = \frac{\pi}{2 K} \sec v + \frac{2 \pi}{K} \sum_{n=0}^{\infty} \frac{ (-1)^n q^{2n+1} }{1-q^{2n+1}}
 \cos((2n+1) v) 
\end{equation}
 given in \cite{DattoliChiccoliLorenzuttaMainoRichettaTorre1992} 
 and in classic texts such as \cite[pp.\ 511--512]{WittakerWatson1958}, 
 writing $ v = \frac{\pi u}{2 K}$ and recalling \eqref{nome}. 
 Following the evaluation technique in Section \ref{subsectiontechnique}, 
 we rewrite the above series expansion for $\text{dc}$ as follows: 
 \begin{equation}\label{dcFouriery}
 \text{{dc}}(2 K w, k) 
 = \frac{\pi}{2 K} \left( \frac{1}{\cos(\pi w)} 
 + 4 \sum_{n=0}^{\infty} \frac{(-1)^n}{e^{(2n+1) \pi \frac{K'}{K} }-1} \cos((2n+1) \pi w) \right). 
\end{equation}
 Setting $w = 0$, \eqref{dcFouriery} this gives us that 
\begin{equation}\label{Bagisformula}
 \frac{K}{2\pi} - \frac{1}{4}
 = \sum_{n=0}^{\infty} \frac{ (-1)^n }{ e^{ (2n+1) \pi \frac{K'}{K} } - 1}, 
\end{equation} 
 and the formula in \eqref{Bagisformula} 
 was also proved by Bagis in \cite{Bagis2019} (Corollary 1), using an identity due to Ramanujan 
 \cite[p.\ 174]{Berndt1991}. 
 Applying term-by-term differentiation to \eqref{Bagisformula} with respect to the elliptic modulus, 
 we may obtain that 
\begin{equation}\label{20221aaaaaaa10AAAA55PLUSSTAR33PM1A}
 \frac{2 K^2 \left(E+\left(k^2-1\right) K \right)}{\pi ^2 (K' (E-K) + E' K)} = 
 \sum _{n=0}^{\infty} (-1)^n (2 n+1) \text{{csch}}^2\left( \frac{2 n + 1}{2} \frac{K'}{K} \pi \right). 
\end{equation}
 We intend to generalize \eqref{20221aaaaaaa10AAAA55PLUSSTAR33PM1A}, 
 by analogy with Zucker's sum in \eqref{Zuckermain}. 

\subsection{Generalizing Zucker's sum}\label{202219999961061110PM1A}
 To generalize Zucker's sum in \eqref{Zuckermain} by replacing 
 $\text{csch}$ with $\text{csch}^{2}$ in \eqref{Zuckermain}, 
 we apply our technique indicated in Section \ref{subsectiontechnique}, as in the below proof. 

\begin{theorem}\label{202211707767171777755AM1AAAA}
 The equality 
 $$ \frac{1}{4}+\frac{2 \left(k^2-1\right) K^3}{\pi ^3} 
 = \sum_{n = 0}^{\infty} \frac{(-1)^n (2 n+1)^2}{e^{ (2 n+1) \frac{K'}{K} \pi } - 1} $$ 
 holds if the above series converges. 
\end{theorem}

\begin{proof}
 We again begin with the Fourier series expansion for $\text{dc}$ in \eqref{qdc}. 
 We rewrite the left-hand side of \eqref{qdc} according to \eqref{dcdefinition}, 
 and then apply the operator $\frac{d^{2}}{du^{2}}$ to both sides of the resultant equality. 
 We then expand the left-hand according to the following differential equations: 
\begin{align}
 \frac{d \, \text{sn} u}{du} & = \text{cn} u \, \text{dn} u, \label{DE1} \\
 \frac{d \, \text{cn }u}{du} & = -\text{sn} u \, \text{dn} u, \ \ \ \text{and} \label{DE2} \\
 \frac{d \, \text{dn} u}{du} & = -k^2 \text{sn} u \, \text{cn} u, \label{DE3} 
\end{align}
 After much simplification and rearrangement using \eqref{DE1}--\eqref{DE3}, we can show that 
\begin{equation}\label{donotrepeat1}
 \frac{\text{dn}(u,k) \left(\text{dn}^2(u, k) - 
 k^2 \text{cn}^2(u, k)\right) \left(\text{cn}^2(u, 
 k)+2 \text{sn}^2(u, k)\right)}{\text{cn}^3(u, k)} 
\end{equation}
 equals 
\begin{equation}\label{donotrepeat2}
 \frac{\pi ^3 \left(3-\cos \left(\frac{\pi u}{K}\right)\right) \sec ^3\left(\frac{\pi u}{2 K}\right)}{16 K^3}-\frac{\pi ^3 
 }{2 K^3} \sum _{n=0}^{\infty} 
 \frac{ (-1)^n q^{2 n+1} (2 n+1)^2 \cos \left(\frac{(2 n 
 + 1) \pi u}{2 K}\right)}{1 - q^{2 n+1}}, 
\end{equation}
 reversing the order of the limiting operations $\frac{d^{2}}{du^{2}}$ 
 and $\sum_{n=0}^{\infty} \cdot$. 
 Following the steps in Section \ref{subsectiontechnique}, we set 
 $u \mapsto 2 w K$. Setting $w = 0$ then gives us an equivalent formulation of the desired result. 
\end{proof}

 We let elliptic singular values be denoted in the usual way \cite[p.\ 298]{BorweinBorwein1987Pi}, writing 
\begin{equation}\label{ellipticsingular}
 \frac{K'(k_{r})}{ K(k_{r}) } = \sqrt{r}, 
\end{equation}
 where the elliptic lambda function is such that 
\begin{equation}\label{ellipticlambda}
 \lambda^{\ast}(r) = k_{r}, 
\end{equation}
 with expressions of the form $K(k_{r})$ admitting explicit symbolic evaluations for natural numbers $r \in \mathbb{N}$. Although the result highlighted as 
 Theorem \ref{202211707767171777755AM1AAAA} 
 is to be applied to obtain new hyperbolic summations, we are to apply this same result using the relations in 
 \eqref{ellipticsingular} and \eqref{ellipticlambda}. 

\begin{example}
 Setting $r = 1$ in \eqref{ellipticsingular}, from the symbolic forms 
\begin{equation}\label{20221102253PM4AAAAApp}
 \lambda^{\ast}(1) = k_{1} = \frac{\sqrt{2}}{2} 
\end{equation}
 and 
\begin{equation}\label{Kk1}
 K(k_{1}) = \frac{\Gamma^{2}\left( \frac{1}{4} \right)}{4 \sqrt{\pi}}, 
\end{equation}
 this gives us, from Theorem \ref{202211707767171777755AM1AAAA}, that 
 $$ \frac{1}{4}-\frac{\Gamma^6 \left(\frac{1}{4}\right)}{64 \pi ^{9/2}} 
 = \sum _{n=0}^{\infty} \frac{(-1)^n (2 n+1)^2}{e^{(2 n+1) \pi }-1}. $$
\end{example}

\begin{example}
 Setting $r = 4$ in \eqref{ellipticsingular}, the symbolic forms
\begin{equation} \label{202219190896710758PM1A}
 \lambda^{\ast}(4) = k_{4} = 3-2 \sqrt{2}, 
\end{equation}
 and 
\begin{equation}\label{Kk4}
 K(k_{4}) = \frac{\left(1+\sqrt{2}\right) \Gamma^2 \left(\frac{1}{4}\right)}{8 \sqrt{2 \pi }} 
\end{equation}
 give us, via Theorem 
 \ref{202211707767171777755AM1AAAA}, that 
$$ \frac{1}{4}-\frac{\left(1+\sqrt{2}\right) \Gamma^6 \left(\frac{1}{4}\right)}{128 \pi ^{9/2}}
 = \sum _{n=0}^{\infty} \frac{(-1)^n (2 n+1)^2}{e^{2 (2 n+1) \pi }-1}. $$
\end{example}

 Theorem \ref{202211707767171777755AM1AAAA} may be further applied 
 so as to obtain the following evaluation for the sum given by 
 setting $s = 3$ and replacing $\text{csch}$ with $\text{csch}^{2}$ 
 in Zucker's sum in \eqref{Zuckermain}. 

\begin{theorem}\label{2022110693177777PM1A777}
 The identity 
 $$ \frac{8 \left(k^2-1\right) K^4 \left(3 E+\left(k^2-3\right) K\right)}{\pi ^4 (K'
 (E-K)+E' K)}
 = \sum _{n=0}^{\infty} (-1)^n (2 n + 1)^3 \text{\emph{csch}}^2\left( \frac{2n+1}{2} \frac{K'}{K} \pi \right) $$
 holds if the above series converges. 
\end{theorem}

\begin{proof}
 This follows in a direct way by applying term-by-term differentiation with respect
 to the elliptic modulus in Theorem \ref{202211707767171777755AM1AAAA}. 
\end{proof}

\begin{example}
 Using the elliptic alpha function 
\begin{equation}\label{ellipticalpha}
 \alpha(r) = \frac{\pi}{ 4 K^{2}(k_{r}) } + \sqrt{r} - \frac{E(k_{r}) \sqrt{r} }{ K(k_{r}) } 
\end{equation}
 together with the elliptic function identity 
\begin{equation}\label{202211069999999932PM1A9999999}
 E' = \frac{\pi}{4 K} + \alpha(r) K 
\end{equation}
 and the known formula $\alpha(1) = \frac{1}{2}$, 
 we may obtain symbolic forms for $E(k_{1})$ and $E'(k_{1})$.
 So, in Theorem \ref{2022110693177777PM1A777}, we set 
 $k = k_{1}$ as the value in \eqref{20221102253PM4AAAAApp}, 
 so as to obtain that 
 $$ \frac{\Gamma^{10} \left(\frac{1}{4}\right)}{128 
 \pi ^{15/2}}-\frac{3 \Gamma^6 \left(\frac{1}{4}\right)}{32 \pi ^{11/2}} 
 = \sum _{n=0}^{\infty} (-1)^n (2 n + 1)^3 
 \text{csch}^2\left(\frac{2 n + 1}{2} \pi \right). $$
\end{example}

\begin{example}
 Setting $k = k_{4}$ as the value in \eqref{202219190896710758PM1A}, 
 from the known valuation $ \alpha(4) = 2 (\sqrt{2} - 1)^2$, 
 this, together with the elliptic alpha function identities in \eqref{ellipticalpha} and \eqref{202211069999999932PM1A9999999}, 
 allows us to obtain, 
 via Theorem \ref{2022110693177777PM1A777}, that 
$$ \frac{\left(2+\sqrt{2}\right) \Gamma^{10} \left(\frac{1}{4}\right)}{1024 
 \pi ^{15/2}}-\frac{3 \left(1+\sqrt{2}\right) \Gamma^6 \left(\frac{1}{4}\right)}{128 \pi ^{11/2}} 
 = \sum _{n=0}^{\infty} (-1)^n (2 n + 1)^3 \text{csch}^2((2 n + 1) \pi). $$ 
\end{example}

 Relative to our proof of Theorem \ref{2022110693177777PM1A777}, 
 we may similarly evaluate 
\begin{equation}\label{202211061050PM1A}
 \sum _{n=0}^{\infty} (-1)^n (2 n + 1)^s \text{{csch}}^2\left( \frac{2n+1}{2} \frac{K'}{K} \pi \right) 
\end{equation}
 for odd natural numbers $s \in \mathbb{N}_{\geq 5}$. 
 It would be desirable to explicitly evaluate \eqref{202211061050PM1A}
 for odd $s \in \mathbb{N}$. This is nontrivial, and may require combinatorial 
 formulas for higher derivatives for $\text{dc}(u, k)$ resulting from repeated applications 
 of the differential equation system given by \eqref{DE1}--\eqref{DE3}. 
 We leave this, together with some other problems considered in Section 
 \ref{sectionconclusion}, for a separate project. 

\subsection{A new sign function}
 Our identities for alternating sums involving factors of the form 
\begin{equation}\label{20221777777771AAAAA06PLUS901PM1A}
 \text{csch}^{2}\left( \frac{2n+1}{2} \frac{K'}{K} \pi \right), 
\end{equation}
 as in Theorem \ref{notfirst} below, 
 are of interest in part because these identities cannot be derived from previously known
 results on sums involving expressions such as 
$$ \text{csch}\left( \frac{2n+1}{2} \frac{K'}{K} \pi \right), $$
 such as the relation whereby $$ \frac{k K}{\pi} 
 = \sum_{n=0}^{\infty} (-1)^{n} \text{csch}\left( \frac{2n+1}{2} \frac{K'}{K} \pi \right) $$
 given as (5.3.4.1) in \cite{PrudnikovBrychkovMarichev1986} 
 and employed in \cite{Yakubovich2018} via term-by-term differentiation with respect to the elliptic modulus. 
 To obtain new sums involving \eqref{20221777777771AAAAA06PLUS901PM1A}, 
 we are to interchange the limiting operations given by the application of $\frac{d}{dk}$ and the 
 application of $\sum_{n=0}^{\infty} \cdot$ 
 in the new result highlighted as Theorem \ref{halfRamanujan} below. 
 To the best of our knowledge, hyperbolic sums 
 involving sign functions such as $ n \mapsto (-1)^{\left\lfloor \frac{n}{2}\right\rfloor } $ 
 have not previously appeared in the relevant literature on hyperbolic sums, 
 which further motivates the research interest in results as in Theorems 
 \ref{halfRamanujan} and \ref{notfirst} below. 

\begin{theorem}\label{halfRamanujan}
 The identity 
 $$ \frac{\left(\sqrt{1-k}+\sqrt{1+k}\right) K }{2 \pi }-\frac{1}{2} 
 = \sum_{n = 0}^{\infty} \frac{ (-1)^{\left\lfloor \frac{n}{2}\right\rfloor }}{e^{ (2 n+1) \frac{K'}{K} \pi }-1} $$
 holds if the above series converges. 
\end{theorem}

\begin{proof}
 From the identity in \eqref{dcFouriery}, we find that the quotient 
 $$ \frac{ \text{dn}(2 K w, k) }{ \text{cn}(2 K w, k) } $$
 admits the same expansion as in \eqref{dcFouriery}. 
 So, setting $w = \frac{1}{4}$ and applying the half-$K$ formulas 
\begin{equation}\label{halfcn}
 \text{cn}\left( \frac{1}{2} K, k \right) 
 = \frac{\sqrt{2} \sqrt[4]{1-k^2} }{ \sqrt{1 + k} + \sqrt{1 - k} } 
\end{equation}
 and 
\begin{equation}\label{halfdn}
 \text{dn}\left( \frac{1}{2} K, k \right) = \sqrt[4]{1-k^2}, 
\end{equation}
 we obtain the expansion $$ \frac{\sqrt{1-k}+\sqrt{k+1}}{\sqrt{2}}
 = \frac{\pi}{2 K} \left(\sqrt{2} + 
 4 \sum _{n = 0}^{\infty} \frac{(-1)^n \cos \left(\frac{2 n + 1}{4} \pi \right)}{ e^{ (2n+1) \frac{K'}{K} \pi } - 1 } \right), $$ 
 which is equivalent to the desired result. 
\end{proof}

\begin{example}\label{exampledcquarter1}
 For the elliptic singular value corresponding to the $r=1$ case in Theorem \ref{halfRamanujan}, 
 we obtain that 
 $$ \frac{\sqrt{2+\sqrt{2}} \Gamma^2 \left(\frac{1}{4}\right)}{8 \pi ^{3/2}}-\frac{1}{2} 
 = \sum _{n=0}^{\infty } \frac{(-1)^{\left\lfloor \frac{n}{2}\right\rfloor }}{e^{(2 n+1) \pi }-1}. $$ 
 This is equivalent to the $q$-digamma evaluation shown below: 
 $$ 4 \pi +\sqrt{\frac{2+\sqrt{2}}{\pi }} \Gamma^2 \left(\frac{1}{4}\right) 
 = -\psi _{e^{8 \pi }} \left(\frac{1}{8}\right)-\psi _{e^{8 \pi }} \left(\frac{3}{8}\right)+\psi _{e^{8 \pi }} \left(\frac{5}{8}\right)+\psi _{e^{8 \pi }} \left(\frac{7}{8}\right). $$
\end{example}

\begin{example}\label{exampledcquarter1p}
 Setting $k = k_{4}$ in Theorem \ref{halfRamanujan}, 
 we may obtain that 
 $$ \frac{\sqrt{3+2 \sqrt{2}+2 \sqrt{4+3 \sqrt{2}}} \Gamma^2 \left(\frac{1}{4}\right)}{16 \pi ^{3/2}}-\frac{1}{2} 
 = \sum _{n=0}^{\infty} 
 \frac{(-1)^{\left\lfloor \frac{n}{2}\right\rfloor }}{e^{2 (2 n+1) \pi }-1}. $$ 
 This is equivalent to the $q$-digamma evaluation shown below: 
\begin{align*}
 & 8 \pi +\sqrt{\frac{3+2 \sqrt{2}+2 \sqrt{4+3 \sqrt{2}}}{\pi }} \Gamma^2 \left(\frac{1}{4}\right) = \\ 
 & -\psi _{e^{16 \pi }} \left(\frac{1}{8}\right) - 
 \psi _{e^{16 \pi }} \left(\frac{3}{8}\right)+\psi _{e^{16 \pi }} \left(\frac{5}{8}\right)+\psi _{e^{16 \pi 
 }} \left(\frac{7}{8}\right). 
\end{align*}
\end{example}

\begin{theorem}\label{notfirst}
 The identity 
\begin{align*}
 & \frac{2 k \left(k^2-1\right) K^2
 \left(\frac{\left(\sqrt{1-k}+\sqrt{1+k}\right) \left(E + 
 \left(k^2-1\right) K \right)}{k 
 \left(k^2-1\right)}-\frac{1}{2} \left(\frac{1}{\sqrt{1+k}}-\frac{1}{\sqrt{1-k}}\right) K\right)}{\pi ^2 (E' 
 K + (E - K) K')} \\ 
 & = \sum _{n=0}^{\infty} 
 (-1)^{\left\lfloor \frac{n}{2}\right\rfloor } (2 n+1) \text{\emph{csch}}^2
 \left( \frac{2n+1}{2} \frac{K'}{K} \pi \right) 
\end{align*}
 holds for suitably bounded $k$. 
\end{theorem}

\begin{proof}
 This follows in a direct way by differentiating both sides of the identity in Theorem \ref{halfRamanujan}
 with respect to the elliptic modulus, and then reversing the order of differentiation and infinite summation, 
 and then applying much simplification. 
\end{proof}
 
 Our series as in Example \ref{20221102339AM1A} are 
 motivated by series evaluations of a similar appearance recorded in \cite{Berndt1978Reine}. 
 For example, the series evaluation 
 $$ \sum_{n=1}^{\infty} (-1)^{n+1} n \text{csch}(\pi n) = \frac{1}{4 \pi} $$
 recorded \cite{Berndt1978Reine} was, as noted in \cite{Berndt1978Reine}, 
 previously proved by many difference authors in 
 \cite{Cauchy1889,Krishnamaghari1920,Nanjundiah1951,NaylorHorsley1931,Sandham1954,Sayer1976,Zucker1979}. 
 To obtain new closed forms from Theorem \ref{notfirst}, we are to make use 
 of the elliptic alpha function \cite[\S5]{BorweinBorwein1987Pi}

\begin{example}\label{202211053AAAAAAAAAAAAAAA2252PLUYS2PMSTAR1A}
 Using the elliptic singular value $k_{1}$, a special case of Theorem \ref{notfirst} gives us that 
 $$ \frac{\sqrt{2+\sqrt{2}} \Gamma^2 \left(\frac{1}{4}\right)}{4 
 \pi ^{5/2}}-\frac{\sqrt{2-\sqrt{2}} \Gamma^6 \left(\frac{1}{4}\right)}{64 \pi ^{9/2}} 
 = \sum _{n = 
 0}^{\infty} (-1)^{\left\lfloor \frac{n}{2}\right\rfloor } (2 n+1) \text{csch}^2\left(\frac{2 n+1}{2} \pi \right). $$ 
\end{example}

\begin{example}\label{20221102339AM1A}
 Using the elliptic singular value $k_{4}$, 
 a special case of Theorem \ref{notfirst} gives us that 
\begin{align*}
 & \frac{\left(\sqrt{1+\sqrt{2}}+\sqrt{2+\sqrt{2}}\right) \Gamma^2 \left(\frac{1}{4}\right)}{16 \pi ^{5/2}} - 
 \frac{\sqrt{4+\sqrt{2} + 2^{7/4}} \Gamma^6 \left(\frac{1}{4}\right)}{256 \pi ^{9/2}} = \\ 
 & \sum _{n=0}^{\infty} (-1)^{\left\lfloor \frac{n}{2}\right\rfloor } (2 n+1) \text{csch}^2((2 n+1) \pi). 
\end{align*}
\end{example}

 As in \cite{BerndtVenkatachaliengar2001}, we record that Nanjundiah's formula 
\begin{equation*}
 \sum_{n=1}^{\infty} \text{csch}^{2}(n \pi) = \frac{1}{6} - \frac{1}{2\pi} 
\end{equation*}
 has, subsequent to Nanjundiah 1951 proof \cite{Nanjundiah1951}, 
 been proved by many different authors, including 
 Berndt \cite{Berndt1977}, Ling \cite{Ling1975}, 
 Kiyek and Schmidt \cite{KiyekSchmidt1967}, 
 Muckenhoupt \cite{Muckenhoupt196465}, and Shafer \cite{Shafer1963}. 
 Nanjundiah's formula, together with our new sums involving $\text{csch}^{2}$ 
 given in Examples \ref{202211053AAAAAAAAAAAAAAA2252PLUYS2PMSTAR1A}
 and \ref{20221102339AM1A}, inspire us to further apply our method 
 in the evaluation of sums involving $\text{csch}^{2}$. 

\subsection{Higher powers of $2n+1$}\label{subsectionpolyhigh}
 We are to again make use of the equality of \eqref{donotrepeat1}
 and \eqref{donotrepeat2}, 
 to prove the following companion to Theorem \ref{202211707767171777755AM1AAAA}. 

\begin{theorem}\label{29999999990229110691900979PM1A}
 The equality 
 $$ \frac{3}{2}+\frac{4 \left( k^2 \left(3 + k' \right) -3 \left(1 + 
 k' \right) \right) K^3}{\left(\sqrt{1-k}+\sqrt{1+k}\right) \pi ^3}
 = \sum _{n=0}^{\infty} \frac{ (-1)^{\left\lfloor \frac{n}{2}\right\rfloor } (2 n+1)^2 }{ e^{ (2 n+1) \frac{ K'}{K} \pi } -1} $$
 holds if the above series converges. 
\end{theorem}

\begin{proof}
 Starting with the equality of \eqref{donotrepeat1} 
 and \eqref{donotrepeat2}, 
 we again set $u \mapsto 2 w K$. 
 Using the half-$K$ formulas in \eqref{halfcn} and \eqref{halfdn}, together with the half-$K$ formula 
 $$ \text{sn}\left( \frac{K}{2}, k \right) = \frac{\sqrt{2}}{\sqrt{1+k} + \sqrt{1-k}}, $$
 we may obtain that 
\begin{align*}
 & \frac{\sqrt{2} \left( k' + 2\right) \left( k' + 1 -k^2 \right)}{\sqrt{1-k}+\sqrt{k+1}} = 
 \frac{3 \pi ^3}{4 \sqrt{2} K^3} - 
 \frac{\pi ^3 }{2 \sqrt{2} K^3} 
 \sum _{n=0}^{\infty} \frac{(2 n+1)^2 (-1)^{\left\lfloor \frac{n}{2}\right\rfloor }}{ e^{ (2 n+1) 
 \frac{K'}{K} \pi } - 1}, 
\end{align*}
 and this is equivalent to the desired result. 
\end{proof}

\begin{example}
 Theorem \ref{29999999990229110691900979PM1A}, together with the elliptic integral singular 
 value $k_{1}$, give us that 
 $$ \frac{3}{2}-\frac{\sqrt{26+17 \sqrt{2}} \Gamma^6 \left(\frac{1}{4}\right)}{64 \pi ^{9/2}}
 = \sum _{n=0}^{\infty} \frac{(2 n+1)^2 (-1)^{\left\lfloor \frac{n}{2}\right\rfloor }}{e^{(2 n+1) \pi }-1}. $$
\end{example}

\begin{example}
 Theorem \ref{29999999990229110691900979PM1A}, 
 together with the elliptic integral singular value $k_{4}$, give us that 
 $$ \frac{3}{2}-\frac{\sqrt{54+37 \sqrt{2}+4 \sqrt{352+249 \sqrt{2}}} \Gamma^6 \left(\frac{1}{4}\right)}{128 \pi ^{9/2}}
 = \sum _{n=0}^{\infty} \frac{(2 n+1)^2 (-1)^{\left\lfloor \frac{n}{2}\right\rfloor }}{e^{2 (2 n+1) \pi }-1}. $$ 
\end{example}

 As below, Theorem \ref{29999999990229110691900979PM1A} 
 may be used to evaluate the $s = 3$ case for the sum 
 obtained from Zucker's sum in \eqref{Zuckermain}
 by replacing $(-1)^{n}$ with $ (-1)^{\left\lfloor \frac{n}{2}\right\rfloor }$
 and by replacing $\text{csch}$ with $\text{csch}^2$. 

\begin{theorem}
 The series 
 $$ \sum _{n=0}^{\infty} 
 (-1)^{\left\lfloor \frac{n}{2}\right\rfloor } (2 n + 1)^3 \text{\emph{csch}}^2 
 \left( \frac{2n+1}{2} \frac{K'}{K} \pi \right) $$ 
 may be evaluated as 
\begin{align*}
 & \Big( 4 k' 
 K^4 \big(6 \big(k \big(k \big(\sqrt{1-k}+\sqrt{1+k}\big)-3 \sqrt{1-k}+3 \sqrt{1+k}\big)-6 \big(\sqrt{1-k}+\sqrt{1+k}\big)\big)
 E + \\
 & \big(36 \big(\sqrt{1-k}+\sqrt{1+k}\big)+k \big(18 \big(\sqrt{1-k}-\sqrt{1+k}\big)+k \big(k \big(k \big(\sqrt{1-k}+\sqrt{1+k}\big)-8 \sqrt{1-k} + \\ 
 & 8
 \sqrt{1+k}\big)-23 \big(\sqrt{1-k}+\sqrt{1+k}\big)\big)\big)\big) K\big) \Big) \Big/ \left( 
 \pi ^4 \big( k' + 1\big) (K' (E-K)+E' K) \right), 
\end{align*}
 if the above series is convergent. 
\end{theorem}

\begin{proof}
 This may be proved by applying term-by-term differentiation, with respect to the elliptic modulus, 
 in Theorem \ref{29999999990229110691900979PM1A}. 
\end{proof}

\begin{example}
 Setting $k = k_1$, we obtain that 
 $$ \frac{\sqrt{218+151 \sqrt{2}} \Gamma^{10} \left(\frac{1}{4}\right) }{512 \pi ^{15/2}}-\frac{3 \sqrt{26 + 
 17 \sqrt{2}} \Gamma^6 \left(\frac{1}{4}\right)}{32 \pi ^{11/2}} = \sum _{n = 
 0}^{\infty} (-1)^{\left\lfloor \frac{n}{2}\right\rfloor } (2 n + 1)^3 \text{csch}^2\left(\frac{2n+1}{2} \pi \right). $$ 
\end{example}

\begin{example}
 Setting $k = k_4$, we obtain that 
 $$ \frac{\sqrt{402+287 \sqrt{2}+4 \sqrt{20296+14358 \sqrt{2}}} 
 \Gamma^{10} \left(\frac{1}{4}\right)}{2048 \pi ^{15/2}}-\frac{3 \sqrt{54+37 \sqrt{2}+4 \sqrt{352+249 \sqrt{2}}} 
 \Gamma^6 \left(\frac{1}{4}\right) }{128 \pi ^{11/2}} $$
 equals 
$$ \sum _{n=0}^{\infty} (-1)^{\left\lfloor \frac{n}{2}\right\rfloor } (2 n + 1)^3 \text{csch}^2((2 n+1) \pi). $$
\end{example}

 We may mimic the above proof to evaluate 
 $$ \sum _{n=0}^{\infty} 
 (-1)^{\left\lfloor \frac{n}{2}\right\rfloor } (2 n + 1)^s \text{{csch}}^2 
 \left( \frac{2n+1}{2} \frac{K'}{K} \pi \right) $$ 
 for odd $s$. 

\section{Applications of the Jacobi elliptic function nc}\label{Mainsectionnc}
 The results introduced in this section on symbolic forms for summations involving the $\text{sech}$
 function are inspired by the formula 
 $$ \frac{\Gamma^2 \left(\frac{1}{4}\right) }{4 \pi ^{3/2}} - \frac{1}{2} 
 = \sum _{n = 1}^{\infty} \text{sech}(\pi n) $$ 
 due to Ramanujan, which is highlighted as an especially amazing formula
 in the Wolfram Mathworld entry on 
 the hyperbolic secant \cite{Weissteinsech}. 
 A Ramanujan summation involving $\text{sech}$ more closely related to our work is Ramanujan's formula (cf.\ \cite{Berndt1977}) 
\begin{equation}\label{Ramanujansimilar}
 \sum_{n=0}^{\infty} (-1)^{n} (2n+1)^{4m-1} \text{sech}\left( \frac{2n+1}{2} \pi \right) = 0. 
\end{equation}
 Our applications of identities as in \eqref{Bagisformula}
 are also inspired by well known Ramanujan formulas such as the following \cite[\S14]{Berndt1989}:
 $$ \sum_{n=1}^{\infty} \frac{n}{e^{2 \pi n} - 1} = \frac{1}{24} - \frac{1}{8 \pi}. $$

 We begin with the following Fourier series expansion \cite[pp.\ 511--512]{WittakerWatson1958}: 
 $$ \text{nc}(u, k) = \frac{\pi}{2 K k'} \sec v - \frac{2 \pi}{K k'} 
 \sum_{n=0}^{\infty} \frac{(-1)^{n} q^{2n+1} }{1 + q^{2n+1}} \cos\left( (2n+1) v \right). $$ 
 Following the steps given in Section \ref{subsectiontechnique}, we obtain that 
 $$ \text{nc}(2 w K, k) 
 = \frac{\pi}{2 k' K } \sec(\pi w) - \frac{2\pi}{ k' K } 
 \sum_{n=0}^{\infty} \frac{(-1)^n}{ e^{ (2n+1) \frac{K'}{K} \pi } + 1 } \cos((2n+1) \pi w). $$ 
 Setting $w = 0$, 
 we can show that the left-hand side of the above equality equals $1$. 
 This gives us an equivalent formulation of the following result: 
\begin{equation}\label{20221102234AM1A}
 \frac{1}{4}-\frac{\sqrt{1-k^2} K }{2 \pi } 
 = \sum _{n=0}^{\infty} \frac{(-1)^n}{e^{ (2 n + 1) \frac{K'}{K} \pi }+1}. 
\end{equation}
 An equivalent result is given as Theorem 6 by Bagis in \cite{Bagis2019}. 
 From this previously known formula, by differentiating with respect to the elliptic modulus, 
 we may obtain the identity 
\begin{equation}\label{7202211076177777111PM1A}
 \frac{2 k' K^{2} 
 (K - E )}{\pi ^2 (K' (E - K ) + 
 E' K )} 
 = \sum _{n=0}^{\infty} 
 (-1)^n (2 n + 1) \text{{sech}}^2\left(\frac{2 n + 1}{2} 
 \frac{K'}{K} \pi \right). 
\end{equation}
 By direct analogy with the material in Section \ref{202219999961061110PM1A}, 
 we may generalize \eqref{7202211076177777111PM1A}
 so as to evaluate 
\begin{equation}\label{202211056111210101010101010PM1A}
 \sum _{n=0}^{\infty} 
 (-1)^n (2 n + 1)^{s} \text{{sech}}^2\left(\frac{2 n + 1}{2} 
 \frac{K'}{K} \pi \right) 
\end{equation}
 for odd $s$. For the sake of brevity, we omit a detailed 
 examination of this, 
 and we leave it do a separate project to 
 obtain an explicit combinatorial formula for 
 \eqref{202211056111210101010101010PM1A} based on repeated applications of the differential equations
 among \eqref{DE1}--\eqref{DE3}. 

 A direct application of the formula in \eqref{20221102234AM1A} 
 that we have proved is given by how this formula
 allows us to prove a stronger version of the formula 
\begin{equation}\label{makestronger}
 2 \sum_{n=0}^{\infty} \frac{(-1)^n}{ e^{(2n+1) \pi} + 1 } 
 + \sum_{n=0}^{\infty} \text{sech}\left( \frac{2n+1}{2} \pi \right) = \frac{1}{2} 
\end{equation}
 highlighted as Corollary 4.22 in \cite{Berndt1978Reine}. This was proved using Euler polynomials in \cite{Berndt1978Reine}. 

\begin{theorem}\label{stronger}
 The following stronger version of Berndt's identity in \eqref{makestronger} holds: 
 $$ \frac{\Gamma^2 \left(\frac{1}{4}\right)}{4 \sqrt{2} \pi ^{3/2}}
 = \sum _{n = 0}^{\infty} \text{\emph{sech}}\left(\frac{2 n+1}{2} \pi \right). $$ 
\end{theorem}

\begin{proof} 
 Using the value for $k_{1}$ in \eqref{20221102253PM4AAAAApp}
 together with the value for $K(k_{1})$ shown in \eqref{Kk1}, 
 the identity in \eqref{20221102234AM1A} gives us that 
 $$ \frac{1}{4}-\frac{\Gamma^2 \left(\frac{1}{4}\right)}{8 \sqrt{2} \pi ^{3/2}} 
 = \sum _{n=0}^{\infty} \frac{(-1)^n}{e^{ (2 n+ 1)\pi}+1}, $$ 
 so that the desired result then follows from \eqref{makestronger}. 
\end{proof}

 Our strengthening of Corollary 4.22 in \cite{Berndt1978Reine}, as in Theorem \ref{stronger}, has not appeared in past publications influenced by 
 \cite{Berndt1978Reine}, which have been based in areas such as number theory 
 \cite{Alkan2015,Dieter1984,HuKimKim2016,Knopp1980,LiuZhang2007,Meyer1997,PengZhang2016,PettetSitaramachandrarao1987,Simsek2003,Zhang2001} 
 mathematical physics \cite{Simsek2010}, high energy physics \cite{Ghilencea2005}, and nuclear physics \cite{DowkerKirsten2002}. The many areas in 
 mathematics and physics related to the discoveries on hyperbolic infinite sums given by Berndt \cite{Berndt1978Reine} inspire the development of further 
 results as in Theorem \ref{stronger}. Since Berndt's formula in \eqref{makestronger} is highlighted as a Corollary in \cite{Berndt1978Reine}, we find it 
 appropriate to highlight the following equivalent formulation of Theorem \ref{stronger} as a Corollary. 

\begin{corollary}
 The $q$-digamma evaluation 
 $$ \psi _{e^{\pi }} \left(\frac{1 - i}{2} \right)-\psi _{e^{\pi }} \left(\frac{1 + i}{2} \right)
 = -i \pi -\frac{i \Gamma^2 \left(\frac{1}{4}\right)}{4 \sqrt{2 \pi }} $$
 holds true.
\end{corollary}

\begin{proof}
 This is almost immediately equivalent to Theorem \ref{stronger}, 
 from the definition in \eqref{qdigamma}. 
\end{proof}

 Many of our results, with a particular regard toward Theorem \ref{stronger}, are of a similar appearance relative to the formula $$ \sum_{n = 
 0}^{\infty} \text{sech}^{2}\left( \frac{ 2 n + 1 }{2} \pi \right) = \frac{1}{\pi} $$ recorded in \cite{Berndt1978Reine}. As in \cite{Berndt1978Reine}, we 
 record that the above formula has been proved by many different authors
 \cite{KiyekSchmidt1967,Ling1975,Nanjundiah1951,Zucker1979}. 

\begin{theorem}\label{theoremncnonhyp}
 The identity $$ \frac{1}{2}-\frac{ \sqrt{k'} \left(\sqrt{1-k}+\sqrt{1+k}\right) K }{2 \pi } = \sum _{n = 0}^{\infty} \frac{(-1)^{\left\lfloor 
 \frac{n}{2}\right\rfloor }}{e^{ (2n+1) \frac{K'}{K} \pi }+1} $$ holds if the above series converges. 
\end{theorem}

\begin{proof}
 Setting $w = \frac{1}{4}$ in the above series expansion for $ \text{nc}(2 w K, k)$, we find that $ \frac{1}{\text{cn}\left( \frac{k}{2},k\right)} $ admits 
 the following expansion: $$ \frac{ \sqrt{2} \pi }{ \sqrt{1 - k^2} K } \left( \frac{1}{2} - \sum_{n=0}^{\infty} \frac{ (-1)^{ \left\lfloor \frac{n}{2} 
 \right\rfloor } }{ e^{ (2n+1) \frac{K'}{K} \pi } + 1 } \right). $$ According to the half-$K$ identity shown in \eqref{halfcn}, 
 we find that $$ \frac{\sqrt{1-k}+\sqrt{1+k}}{\sqrt{2} \sqrt{k'} } $$ 
 admits the same expansion, and this gives us an equivalent version of the desired result. 
\end{proof}

\begin{example}
 The symbolic form $$ \frac{1}{2}-\frac{\sqrt{1+\sqrt{2}} \Gamma^2 \left(\frac{1}{4}\right)}{8 \pi ^{3/2}} = \sum _{n=0}^{\infty} 
 \frac{(-1)^{\left\lfloor \frac{n}{2}\right\rfloor }}{e^{ (2 n + 1) \pi }+1} $$ holds and may be proved using Theorem \ref{theoremncnonhyp} together with 
 the elliptic singular value indicated in \eqref{Kk1}. This is equivalent to the $q$-digamma evaluation 
 $$ -4 \pi -\sqrt{\frac{1+\sqrt{2}}{\pi }} \Gamma^2 \left(\frac{1}{4}\right) 
 = \psi _{e^{8 \pi }} \left(\frac{1 - i}{8} \right)+\psi _{e^{8 \pi }} \left(\frac{3 - i}{8} \right) - 
 \psi _{e^{8 \pi }} \left(\frac{5 - i}{8} \right)-\psi_{e^{8 \pi }} \left(\frac{7 - i}{8} \right). $$ 
\end{example}

 Ramanujan's closed forms for 
\begin{equation}\label{Ramanujan4t}
 \sum_{n = 0}^{\infty} 
 (-1)^n (2 n + 1)^{-4t - 1} \text{sech}\left( \frac{2n+1}{2} \pi \right) 
\end{equation}
 have been considered, as in \cite{Berndt1978Mag}, as especially notable contributions out of the results given in Ramanujan's notebooks. 
 This further motivates our generalizations of or related to Ramanujan's $\text{sech}$-sums as in 
 \eqref{Ramanujanmainmotivate} and \eqref{Ramanujan4t}. 

\begin{example}
 The symbolic form 
$$ \frac{1}{2}-\frac{\sqrt{\frac{1}{2} \left(2 \sqrt{2}+\sqrt{4+3 \sqrt{2}}\right)} 
 \Gamma^2 \left(\frac{1}{4}\right)}{8 \pi ^{3/2}}
 = \sum_{n = 0}^{\infty} \frac{(-1)^{\left\lfloor \frac{n}{2}\right\rfloor }}{e^{2 (2 n+1) \pi }+1} $$ 
 holds and may be proved using Theorem \ref{theoremncnonhyp} 
 together with the elliptic singular value indicated in \eqref{Kk4}. 
 We may rewrite this a $q$-digamma evaluation in much the same way as before. 
\end{example}

 Ramanujan's many evaluations for infinite hyperbolic sums served as a main source of motivation behind Xu's work in \cite{Xu2018}. In particular, 
 Ramanujan's evaluations whereby $$ \sum_{n = 0}^{\infty} (2 n + 1)^2 \text{sech}\left( \frac{2n+1}{2} \pi \right) = \frac{\Gamma^6 
 \left(\frac{1}{4}\right)}{16 \sqrt{2} \pi ^{9/2}} $$ and $$ \sum _{n=0}^{\infty} (2 n + 1)^2 \text{sech}^2\left( \frac{2n+1}{2} \pi \right) = 
 \frac{\Gamma^8 \left(\frac{1}{4}\right)}{192 \pi ^6} $$ were highlighted as a main sources of motivation concerning the results introduced by 
 Xu in \cite{Xu2018}. 

 By applying term-by-term differentiation with respect to the elliptic modulus in Theorem \ref{theoremncnonhyp}, we can show that the expression 
\begin{align*}
 & \big( K^2 \sqrt{k'} \big(\big(k \big(\sqrt{1-k} - 
 \sqrt{1+k}\big)+2 \big(\sqrt{1-k}+\sqrt{1+k}\big)\big) K-2
 \big(\sqrt{1-k}+\sqrt{1+k}\big) E\big) \big) \big/ \\ 
 & \big(\pi ^2 (K' (E-K) + E' K) \big) 
\end{align*}
 equals $$ \sum _{n=0}^{\infty} 
 (-1)^{\left\lfloor \frac{n}{2}\right\rfloor } (2 n+1) \text{{sech}}^2 \left( \frac{2n+1}{2} \frac{K'}{K} \pi \right) $$ 
 if the above series converges. By direct analogy with the material in Section \ref{subsectionpolyhigh}, 
 we can mimic our above derivation so as to evaluate $$ \sum _{n=0}^{\infty} 
 (-1)^{\left\lfloor \frac{n}{2}\right\rfloor } (2 n+1)^{s} \text{{sech}}^2 \left( \frac{2n+1}{2} \frac{K'}{K} \pi \right) $$ 
 for odd $s$. 

\section{Powers of Jacobi elliptic functions}\label{sectionpowers}
 Using \eqref{nssquared} together with the technique in Section \ref{subsectiontechnique}, we can prove that $$ -\frac{K^2 \left(2 E K k' + E^2 + 
 K^2 \left(k^2 (k'+1)-2 k'-1\right)\right)}{4 \pi ^3 (K' (E-K) + E' K)} $$ equals $$ \sum _{n=1}^{\infty} (-1)^n n^2 \text{{csch}}^2\left(\frac{2 n 
 \pi K'}{K}\right) $$ if the above series converges. For example, this gives us that $$ -\frac{1}{32 \pi ^2}-\frac{\left(1+\sqrt{2}\right) 
 \Gamma^4 \left(\frac{1}{4}\right)}{128 \pi ^4}+\frac{\left(1+\sqrt{2}\right) \Gamma^8 \left(\frac{1}{4}\right) }{2048 \pi ^6} = \sum _{n = 
 1}^{\infty} (-1)^n n^2 \text{csch}^2(2 n \pi). $$ We leave it to a separate project to generalize this result and to investigate the use of our method 
 together with Fourier series expansions apart from \eqref{nssquared} that involve powers of Jacobi elliptic functions. 

\section{Further considerations}\label{sectionconclusion}
 The computation of the inverse functions among \eqref{definitionsn}--\eqref{definitiondn} 
 often turns out to be difficult, even with state-of-the-art Computer Algebra Systems \cite{Batista2019}. 
 For example, there is no direct way of 
 using built-in Mathematica commands such as 
 {\tt JacobiSN}, {\tt JacobiCN}, or {\tt JacobiDN} 
 to compute the ``actual'' Jacobi elliptic functions defined via \eqref{definitionsn}--\eqref{definitiondn} relative to the 
 Mathematica definition for the incomplete elliptic $F$-integral
\begin{equation*}
 \int_{0}^{\phi} \frac{dt}{\sqrt{1 - k \sin^2 t}} 
\end{equation*}
 with arguments $\phi$ and $k$, which is in contrast to \eqref{actualF}. This kind of practical computational problem motivates the development of 
 new and efficient ways of expressing and applying the Jacobi elliptic functions in the form of series expansions involved in this article. 

 Ramanujan's identity in \eqref{instanceRamanujan}  was applied by Berndt in \cite{Berndt2016}  to prove the following remarkable closed form: 
 $$ \int_{-\infty}^{\infty} \frac{dx}{ \cos \sqrt{x} + \cosh \sqrt{x}}  = \frac{\pi}{4} \frac{ \Gamma^{2}\left( \frac{1}{4} \right) }{ \Gamma^{2}\left( 
 \frac{3}{4} \right)}. $$ How can we obtain similar results using our evaluations related to 
 Ramanujan's formula in \eqref{instanceRamanujan}, as in Section \ref{Mainsectionnc}? 
 We leave this for a separate project. 
 Also, we leave it to a separate project to apply our technique using 
 identities such as the third-$K$ formula, as opposed to the half-$K$ formula we have applied. 

\subsection*{Acknowledgements}
 The author wants to express his sincere thanks 
 to Alexey Kuznetsov for many useful discussions.

 \

\noindent John M.\ Campbell

\noindent Department of Mathematics

\noindent Toronto Metropolitan University

\noindent {\tt jmaxwellcampbell@gmail.com}

\end{document}